\documentclass{amsart}%
\usepackage{amsfonts}
\usepackage{amsmath}
\usepackage{todonotes}
\usepackage{amssymb}
\usepackage[margin=1.6in]{geometry}
\usepackage{graphicx}
\usepackage{xcolor}
\usepackage[all,cmtip]{xy}
\usepackage[backref]{hyperref}%
\newtheorem{theorem}{Theorem}
\theoremstyle{plain}
\newtheorem*{acknowledgement}{Acknowledgement}

\newtheorem{lemma}{Lemma}

\newtheorem{proposition}{Proposition}
\theoremstyle{remark}
\newtheorem{remark}{Remark}
\newtheorem{example}{Example}
\newtheorem{variant}{Variant}

\numberwithin{equation}{section}
\begin{document}
\title[Realizations for the $K$-theory of varieties]{The standard realizations for the\linebreak$K$-theory of varieties}
\author{Oliver Braunling}
\address{University of Freiburg, Germany}
\email{oliver.braunling@gmail.com}
\author{Michael Groechenig}
\address{Department of Mathematics, University of Toronto, Canada}
\email{michael.groechenig@utoronto.ca}
\author{Anubhav Nanavaty}
\address{University of California, Irvine, USA}
\email{nanavaty@uci.edu}

\thanks{O. B. was supported by DFG GK1821 \textquotedblleft Cohomological Methods
in Geometry\textquotedblright\ . M. G. was supported by NSERC Discovery Grant RGPIN-2019-05264. A. N. was supported by DMS-1944862.}

\begin{abstract}
The Grothendieck ring of varieties has well-known realization maps to, say,
mixed Hodge structures or compactly supported $\ell$-adic cohomology.
Zakharevich and\ Campbell have developed {a spectral
refinement} of the Grothendieck ring of varieties. We develop a realization map to Voevodsky
mixed motives, and this lifts the standard realizations of motives to this
setting, at least over perfect fields which have resolution of singularities.
\end{abstract}
\maketitle
%

%%%%%%%%%%%%% DER KANN DANN SPÄTER WEG ABER DU HAST MAL GESAGT FÜR DIE ARBEIT AN EINEM PAPER HAST DU DAS LIEBER DA
%%%%\tableofcontents
%%%%%%%%%%%%%%%%%%%%%%%%%%%%%%%%%%%%%%%%

\section{Motivation}

Let $k$ be a perfect field. We write $K_{0}(\mathcal{V}_{k})$ for the
Grothendieck ring of varieties. There are the standard motivic realizations, also known as motivic measures,
like $\ell$-adic%
\[
K_{0}(\mathcal{V}_{k})\longrightarrow K_{0}(\operatorname*{Rep}%
\nolimits_{\operatorname*{Gal}(k^{\operatorname*{sep}}/k)}(\mathbb{Q}_{\ell}))
\]
(for $\ell \in k^{\times}$) or to mixed Hodge structures, for example in the format of Hodge--Deligne
polynomials%
\begin{align*}
K_{0}(\mathcal{V}_{k})  & \longrightarrow\mathbb{Z}[u,v]\\
X  & \longmapsto\sum_{p,q\in\mathbb{Z}}(-1)^{p+q}h^{p,q}(H_{c}^{p+q}%
(X_{\mathbb{C}},\mathbb{C}))\cdot u^{p}v^{q}%
\end{align*}
based on the Hodge numbers of compactly supported cohomology. The purpose of
this paper is to extend such realizations beyond $K_0$. For the Betti and \'{e}tale realization this problem was already solved by Campbell, Wolfson and Zakharevich in \cite{MR3993931}. In this paper we develop an axiomatic approach: \textit{Every invariant of smooth varieties which satisfies $h$-descent and $\mathbb{A}^1$-invariance gives rise to a canonical realization}.

There are two approaches to extend the Grothendieck ring of varieties to a ring spectrum $K(\mathcal{V}_{k})$ such that
\[
\pi_{0}K(\mathcal{V}_{k})=K_{0}(\mathcal{V}_{k})\text{,}%
\]
where the latter refers to the classical definition of the Grothendieck ring of
varieties. Setting up such a spectral refinement was pioneered by Zakharevich \cite{zaka,zakb}. There is also a different approach by Campbell \cite{MR3955537}. That both methods agree, at least on the level of spaces, was shown in \cite{czpreprint}.
Our realizations will be set up as maps of spectra mapping out of Campbell's model.

We construct the realizations in two steps: (1) First we construct a realization to geometric mixed motives in the sense of Voevodsky. (2) Then we use that invariants satisfying $h$-descent and $\mathbb{A}^1$-invariance pin down realization functors for such motives. Composing both maps produces the desired realizations.

%As our main technical result we set up a realization map to geometric mixed
%motives in the Nisnevich (or \'{e}tale) topology.

Let $A$ be a commutative unital ring, serving as our ring of coefficients. We
write $\mathcal{DM}_{gm,t}^{eff}(k;A)$ for the $A$-linear DG\ category of
geometric mixed motives in the topology $t\in\{\acute{e}t,Nis\}$. We extract a Waldhausen category from such a DG category such that the weak
equivalences are quasi-isomorphisms in the DG\ sense.

 Suppose $\mathcal{V}%
_{k}$ denotes the $SW$-category of $k$-varieties (see
\S \ref{sect_RunningConventions} for details).

\begin{theorem}
Suppose $k$ is a perfect field. If $k$ has positive characteristic $p>0$, we
assume that $\frac{1}{p}\in A$. We construct a weakly $W$-exact functor
$F=(F_{!},F^{!},F^{w})$ from $\mathcal{V}_{k}$ to the Waldhausen DG category
$\mathcal{DM}_{gm,t}^{eff}(k;A)$. 
{In particular, it
induces a map of spectra%
\[
K(F)\colon K(\mathcal{V}_{k})\longrightarrow K(\mathcal{DM}_{gm,t}%
^{eff}(k;A))\text{.}%
\]}
Moreover, for any variety $X\in
\mathcal{V}_{k}$ we have%
\[
F(X)=M^{c}(X)
\]
in the homotopy category $\mathrm{DM}_{gm,Nis}^{eff}(k;A)$, i.e. $F(X)$
represents the motive with compact support attached to $X$. 
\end{theorem}

See Theorem \ref{thm_MixMotRealization}. This result does not rely on the resolution of singularities by relying on Kelly's techniques in positive characteristic.

The above motivic realization leads to the following very
general mechanism to produce realizations. { For every DG enhanced \emph{realization functor} (or quasi-functor) 
$$\mathcal{DM}_{gm,t}^{eff}(k;A) \to \mathcal{D}$$
to a DG category $\mathcal{D}$, we obtain a realization map
$$K(\mathcal{V}_{k}) \to K(\mathcal{DM}_{gm,t}^{eff}(k;A)) \to K(\mathcal{D})\text{.}$$
The DG category $\mathcal{DM}_{gm,t}^{eff}(k;A)$ possesses a universal property, to be described below, which simplifies the construction of realization functors.}
Suppose $k$ is a perfect field
which admits resolution of singularities and $A$ as in the previous theorem.
Suppose $\mathcal{D}$ is a triangulated cocomplete and compactly
genererated $A$-linear DG category. Suppose%
\[
\mathcal{D}_{finite}\subseteq\mathcal{D}%
\]
is a suitable subcategory (we list precise conditions in the main body of the
paper). Following Vologodsky, we write $\mathcal{T}^{h,\Delta}$ for
DG\ quasi-functors on smooth $k$-varieties (with finite correspondences with
$A$-coefficients as morphisms) which satisfy $\mathbb{A}^{1}$-invariance and
$h$-descent.
%%The {latter requires resolution of singularities}.

\begin{theorem}
Suppose $k$ is a perfect field for which we have resolution of singularities
and of finite cohomological dimension with respect to $A$-coefficients in the
sense of Equation \ref{l_Cond_GaloisCohomDim}. Suppose we are given a DG
quasi-functor%
\[
\phi\in\mathcal{T}^{h,\Delta}(A[Sm_{k}],\mathcal{D}_{finite})\text{.}%
\]
Then there is a map of spectra%
\[
\operatorname*{R}\nolimits^{\phi}\colon K(\mathcal{V}_{k})\longrightarrow
K(\mathcal{D}_{finite})
\]
such that the following hold:

\begin{enumerate}
\item Suppose $X$ is a smooth proper $k$-variety. In $K_{0}$ we get%
\[
\operatorname*{R}\nolimits^{\phi}([X])=[\phi(X)]
\]
and if $f\colon X\overset{\sim}{\rightarrow}X$ is any automorphism, we get in
$K_{1}$ that%
\[
\operatorname*{R}\nolimits^{\phi}([f\colon X\overset{\sim}{\rightarrow
}X])=\left[  \phi(f)\colon\phi(X)\overset{\sim}{\rightarrow}\phi(X)\right]
\text{.}%
\]

\item Suppose $X$ is a smooth $k$-variety, $\overline{X}$ a smooth
compactification with a smooth closed subvariety $Z\subseteq\overline{X}$ such
that $X=\overline{X}\setminus Z$. Then in $K_{0}$,%
\[
\operatorname*{R}\nolimits^{\phi}([X])=[\phi(\overline{X})]-[\phi(Z)]
\]
and if one can extend the automorphisms such that%
\[
\xymatrix{
Z \ar[d]_{f} \ar@{^{(}->}[r] & \overline{X} \ar[d]^{f} \\
Z \ar@{^{(}->}[r] & \overline{X}
}
\]
commutes, then $\operatorname*{R}\nolimits^{\phi}([f\colon X\overset{\sim
}{\rightarrow}X])=\operatorname*{R}\nolimits^{\phi}([f\colon\overline
{X}\overset{\sim}{\rightarrow}\overline{X}])-\operatorname*{R}\nolimits^{\phi
}([f\mid_{Z}\colon Z\overset{\sim}{\rightarrow}Z])$.
\end{enumerate}
\end{theorem}

See Theorem \ref{thm_GeneralRealizationTheorem}. Realizations like, for example, the

\begin{itemize}
\item Betti realization to finitely generated $A$-modules for $A$ some
Noetherian ring,

\item mixed Hodge realization,

\item $\ell$-adic \'{e}tale realizations,

\item or more broadly any mixed Weil cohomology theory in the sense of Cisinski and D\'{e}glise,
\end{itemize}

may be used as input for the above theorem, producing corresponding realizations for
$K(\mathcal{V}_{k})$.\

This is not the first construction of realizations for
$K(\mathcal{V}_{k})$. The Betti and $\ell$-adic realization were constructed by Campbell,
Wolfson and Zakharevich in \cite{MR3993931} and we crucially use their device
of weakly $W$-exact functors. That we add a mixed Hodge realization solves \textit{Problem 7.4} loc. cit. and, at least for fields admitting resolution of singularities, also confirms the expectation mentioned loc. cit. that other Weil cohomology theories admit realizations through the Cisinski--D\'{e}glise device.

The idea to extend realization functors uniquely to motives by demanding
them to exist on smooth varieties and have suitable descent and $\mathbb{A}%
^{1}$-invariance properties can be found in work of Vologodsky
\cite{MR3004172} and Robalo {\cite{robalo}}. These ideas have a long series of predecessors constructing
various realizations in the setting of mixed motives, among others by Huber
\cite{MR1439046, MR1775312} or Ivorra \cite{MR2377240, MR3518311}.
For the purposes of this paper, we follow the framework of Vologodsky. Vologodsky at times assumes $\operatorname{char}(k)=0$, but as long as we use coefficients with $p$ invertible once $p>0$, he really only demands this to be able to use resolution of singularities.

Regarding mixed motives, we will follow Voevodsky's theory, most of what we need is explained in the book \cite{MR2242284}, with three
noteworthy exceptions:

\begin{itemize}
\item Not all necessary theorems are actually proven in the book. However, in
these cases we just refer to the original literature. This is unproblematic.

\item In characteristic $p>0$ we would like to use results which, in
Voevodsky's theory, are only available under the assumption that resolution of
singularities will be established also in positive characteristic. Should this
ever occur, the better. However, since at present this is not available, we
instead use Kelly's work based on alterations \cite{MR3673293}. This comes
with the price of having to invert $p$.

\item The original literature as well as the book \cite{MR2242284} only
develop mixed motives on the level of triangulated categories. However, we
crucially need a DG\ enhancement. This has been worked out by Beilinson and
Vologodsky in \cite{MR2399083}.
\end{itemize}

%%%% ACKNOWLEDGEMENT
\begin{acknowledgement}
The first author thanks Brad Drew. The third author would like to thank his advisor, Jesse Wolfson, for insightful discussions and suggestions.
\end{acknowledgement}
%%%% ACKNOWLEDGEMENT

\subsection{Running conventions\label{sect_RunningConventions}}

The word \emph{scheme} will always refer to a separated scheme over a field.
We stress this because both sources \cite{MR1764199} and \cite{MR2242284} also
use such assumptions tacitly (they both announce early on that all their
schemes will satisfy conditions which are implied by the above).

By a $k$-\emph{variety} we mean a reduced finite type (separated) scheme over
the field $k$. This is in line with the usage in \cite[Def. 2.1]{MR3955537},
specialized to a base field. Morphisms between $k$-varieties will always be
assumed to be $k$-morphisms.

For complexes we use the terms \emph{bounded below/above} with reference to
homological indexing, e.g. $\cdots\rightarrow C_{1}\rightarrow
C_{0}\rightarrow0$ is bounded above, even though the indices themselves are
bounded from below. This is the same convention as is used in \cite{MR2242284}.

\section{Realization to Mixed Motives\label{sect_RealizationToMixedMotives}}

\subsection{Introduction}

We set up a map%
\[
K(\mathcal{V}_{k})\longrightarrow K(\mathrm{C}^{gm})\text{,}%
\]
where $\mathrm{C}^{gm}$ is the Waldhausen category of geometric effective
mixed motives (we give a precise definition later). The key difficulty is to
attach to each smooth variety $X$ in $\mathcal{V}_{k}$ a complex which is
strictly functorial in both closed immersions as well as open immersions. To
this end, we use the $z_{equi}(X,0)$ cycle sheaves. They are a concrete
representative of the compactly supported motive $M^{c}(X)$.

\subsection{Preliminary remarks on DG\ categories}\label{prelim}

{ For a DG category $\mathcal{A}$ we will denote by $\mathrm{mod}(\mathcal{A}^{\rm op})$ the DG category of \emph{right} DG $\mathcal{A}$-modules. The full subcategory of compact objects (or perfect DG modules) will be denoted by $\mathrm{perf}(\mathcal{A}^{\rm op})$. There is a Yoneda-style quasi-embedding
\begin{equation}\label{eqn:yoneda}
\mathcal{A} \to \mathrm{perf}(\mathcal{A}^{\rm op}).
\end{equation}
A DG category is called \emph{triangulated} if the functor \eqref{eqn:yoneda} is a quasi-equivalence. The latter implies that
$$Ho(\mathcal{A}) \to Ho(\mathrm{perf}(\mathcal{A}^{\rm op}))$$
is an equivalence of categories.
\begin{remark}\label{convention}
By abuse of language we will denote $X \in \mathcal{A}$ and the associated DG module in $\mathrm{perf}(\mathcal{A}^{\rm op})$ by $X$. Similarly, we will at times gloss over the difference between a triangulated DG category $\mathcal{A}$ and $\mathrm{perf}(\mathcal{A}^{\rm op})$.
\end{remark}
}

{The content of the following proposition is discussed at the beginning of \cite[Section 3]{toen}.

\begin{proposition}
\label{prop_ModelStruct}%
Let $\mathcal{A}$ be a DG category. There is a cofibrantly generated model structure on $\mathrm{mod}(\mathcal{A}^{\rm op})$, such that a morphism 
$$F \to G$$
of right DG modules is a weak equivalence (resp. a fibration) if and only if for every $X \in \mathcal{A}$ the induced morphism of complexes
$$F(X) \to G(X)$$
is a weak equivalence (respectively a fibration).
\end{proposition}

We therefore have the structure of a cofibrantly generated model category on $$\mathrm{mod}(\mathcal{A}).$$ This structure will play an important yet transient role in Subsection \ref{functorial}.

Following \cite[Subsection 5.2(a)]{toen} we define the Waldhausen category associated to a DG category $\mathcal{A}$ to be
$$\mathrm{perf}_{\rm co}(\mathcal{A}^{\rm op}),$$
i.e. the full subcategory of cofibrant and compact right DG modules. The aforementioned model category structure endows $\mathcal{A}$ with the requisite class of cofibrations and weak equivalences. The algebraic $K$-theory space $K(\mathcal{A})$ is defined to be the Waldhausen $K$-theory of $\mathrm{perf}_{\rm co}(\mathcal{A}^{\rm op})$. The embedding \eqref{eqn:yoneda} factors through $\mathrm{perf}_{\rm co}(\mathcal{A}^{\rm op})$ (\cite[p. 630]{toen}).
}

We shall use calligraphic letters, as in $\mathcal{C}$, for a DG\ category and
roman letters, as in $\mathrm{C}$, to denote { classical (e.g. triangulated)} categories. In
particular, if $\mathcal{C}$ is a triangulated DG\ category, we can simply
write%
\[
\mathrm{C}:=Ho(\mathcal{C})
\]
for its homotopy category. This is compatible with the notation in
\cite{MR3004172}.

\subsection{Recollections}

Let $A$ be a commutative unital ring. It will serve as our ring of
coefficients, e.g. it could be $A:=\mathbb{Z}$.

We fix a perfect field $k$ such that there exists some $N$ such that for all
$r>N$ we have%
\begin{equation}
H^{r}(\operatorname*{Gal}(k^{\operatorname*{sep}}%
/k),M)=0\label{l_Cond_GaloisCohomDim}%
\end{equation}
for all $A$-modules $M$.

\begin{example}
\label{example_CohomDim}This condition is usually harmless.

\begin{enumerate}
\item If $k$ is separably closed, this condition is automatically satisfied,
so $k=\mathbb{C}$ is fine.

\item If $k$ has finite strict cohomological dimension (in the classical sense
of Galois cohomology), the condition is satisfied.

\item In particular, the condition is met if $k$ is a finite field.

\item For {$A=\mathbb{Z}$} and $k=\mathbb{R}$ the condition is \emph{not}
satisfied because, in positive degrees, the cohomology is periodic of period
$2$.
\end{enumerate}
\end{example}

Let $Sm_{k}$ be the category of smooth separated $k$-varieties and
$k$-morphisms (this is the notation of Beilinson and Vologodsky
\cite{MR2399083}, \cite{MR3004172}; the book \cite{MR2242284} uses the
marginally different notation $Sm/k$).

Next, $A_{tr}[Sm_{k}]$ denotes the category of finite correspondences over
$k$. It has the same objects as $Sm_{k}$, but instead of genuine
$k$-morphisms, we consider finite correspondences with coefficients in the
ring $A$ (this is the notation of Beilinson and Vologodsky; the book
\cite{MR2242284} has $A=\mathbb{Z}$ instead and denotes the same category by
$Cor_{k}$).

\subsection{DG\ category of mixed motives}

Let $t\in\{\acute{e}t,Nis\}$ be either the \'{e}tale or the Nisnevich topology.
Next, one sets up a triangulated DG\ category of effective mixed motives
$\mathcal{DM}_{t}^{eff}(k;A)$ over the base field $k$ and with coefficients in
$A$. For $t=\acute{e}t$ this is described in \cite[\S 2.2, p. 378, last
paragraph]{MR3004172}, and for $t=Nis$ in \cite[Remark 2.6]{MR3004172} and
there is a DG quasi-functor%
\begin{equation}
\mathcal{DM}_{Nis}^{eff}(k;A)\longrightarrow\mathcal{DM}_{\acute{e}t}%
^{eff}(k;A)\text{,} \label{lzz1}%
\end{equation}
also described loc. cit. which corresponds to enforcing \'{e}tale descent as
opposed to the weaker Nisnevich descent. Following the aforementioned
convention, the homotopy categories are denoted by $\mathrm{DM}_{t}%
^{eff}(k;A)$ {and we get an induced triangulated functor}
\begin{equation}
\mathrm{DM}_{Nis}^{eff}(k;A)\longrightarrow\mathrm{DM}_{\acute{e}t}%
^{eff}(k;A)\label{lzz2}%
\end{equation}
{of triangulated categories.}
If $A$ is a $\mathbb{Q}$-algebra, Equation \eqref{lzz1} (and therefore Equation
\eqref{lzz2}) are equivalences.

\begin{example}
As Chow groups do not satisfy \'{e}tale descent, Chow groups are representable
in $\mathcal{DM}_{Nis}^{eff}(k;A)$, but not in $\mathcal{DM}_{\acute{e}%
t}^{eff}(k;A)$.
\end{example}

Next, one defines the DG\ category of \emph{geometric mixed motives}%
\begin{equation}
\mathcal{DM}_{gm,t}^{eff}(k;A):=\mathcal{DM}_{t}^{eff}(k;A)^{perf}
\label{lvv1}%
\end{equation}
as the full DG\ subcategory of $\mathcal{DM}_{t}^{eff}(k;A)$ such that the
objects are compact in the homotopy category $\mathrm{DM}_{t}^{eff}%
(k;A)=Ho(\mathcal{DM}_{t}^{eff}(k;A))$. As a shorthand for later use, we
define 
{
\begin{align}
\mathcal{C}  &  := \mathrm{perf}_{\rm co}(\mathcal{DM}_{Nis}^{eff}(k;A)^{\rm op})\text{,}\nonumber\\
\mathcal{C}^{gm}  &  :=\mathrm{perf}_{\rm co}(\mathcal{DM}_{gm,Nis}^{eff}(k;A)^{\rm op})\text{,} \label{lgm1}%
\end{align}
Since the functor \eqref{eqn:yoneda} induces an equivalence of homotopy categories for triangulated DG categories, the categories above are DG enhancements for the triangulated}
categories%
\begin{align*}
\mathrm{C}=  &  \mathrm{DM}_{Nis}^{eff}(k;A)\text{,}\\
\mathrm{C}^{gm}=  &  \mathrm{DM}_{gm,Nis}^{eff}(k;A)\text{.}%
\end{align*}
These are the triangulated categories of (effective, resp. geometric)
Nisnevich mixed motives, and at least the latter is precisely the same as in
Voevodsky's original formalism. {In the case of the former there is a subtle but irrelevant difference since Beilinson--Vologodsky allow arbitrary unbounded complexes, whereas Voevodsky imposes a boundedness condition.}

\begin{remark}
[Compatibility with Voevodsky's formalism]\label{rmk_CompareVoevodsky}We
should explain the relationship to the category $\mathrm{C}^{gm}$ to
Voevodsky's original theory, as described in the book \cite{MR2242284}. We
pick $t:=Nis$. We have a triangulated equivalence%
\[
\mathbf{DM}_{gm}^{eff}(k,A)\cong\mathrm{DM}_{gm,Nis}^{eff}(k;A)\text{,}%
\]
where $\mathbf{DM}_{gm}^{eff}(k,A)$ refers to the boldface notational
conventions used in \cite[Lecture 14, Definition 14.1]{MR2242284}. In short:
The above consideration describe a DG enhancement of Voevodsky's original
category of geometric motives. Let us explain the comparison: In
\cite{MR2242284}, following Voevodsky's original works, one first sets up a
category of effective Nisnevich motives $\mathbf{DM}_{Nis}^{eff,-}(k,A)$ as
the $\mathbb{A}^{1}$-localization of the triangulated category of bounded
above complexes of Nisnevich sheaves {(note that this is called $\mathbb{Z}_{tr}[Sm_{k}]$ by Beilinson--Vologodsky)}%
\begin{equation}
\mathbf{DM}_{Nis}^{eff,-}(k,A)=D^{-}\left(  Sh_{Nis}(Cor_{k},A)\right)
[W_{\mathbb{A}}^{-1}]\text{,} \label{lkj1}%
\end{equation}
where $W_{\mathbb{A}}$ denotes the class of $\mathbb{A}^{1}$-equivalences
$X\rightarrow X\times\mathbb{A}^{1}$. This category is \emph{genuinely
different} from the triangulated category $\mathrm{DM}_{Nis}^{eff}(k;A)$ above
from the Beilinson--Vologodsky setting. However, the difference is merely that%
\[
\mathbf{DM}_{Nis}^{eff,-}(k,A)\longrightarrow\mathrm{DM}_{Nis}^{eff}(k;A)
\]
is essentially the inclusion of the bounded above complexes into all unbounded
complexes (see \cite[Remark 2.6 (and below)]{MR3004172} or the introductory
explanations in \cite[\S 2, before \S 2.1]{MR2399083}). So these categories
are actually different. Then the book \cite{MR2242284} defines%
\[
\mathbf{DM}_{gm}^{eff}(k,A)\subset\mathbf{DM}_{Nis}^{eff,-}(k,A)
\]
as the thick triangulated subcategory generated by motives
$M(X):=L_{\mathbb{A}^{1}}(A_{tr}(X))$ of smooth $k$-varieties $X$ (so
$A_{tr}(X)=\mathbb{Z}_{tr}(X)$ in the notation of \cite[Definition
2.8]{MR2242284} for $A=\mathbb{Z}$). However, one can alternatively
characterize $\mathbf{DM}_{gm}^{eff}(k,A)$ as the compact objects inside the
triangulated category $\mathbf{DM}_{Nis}^{eff,-}(k,A)$, see \cite[Theorem 6.2,
applied as in Example 6.3, for $S:=k$]{MR2529161}. Hence, the definition in
Equation \ref{lvv1} yields a DG enhancement of the same triangulated category
as is Voevodsky's original $\mathbf{DM}_{gm}^{eff}(k,A)$, i.e.%
\[
\mathbf{DM}_{gm}^{eff}(k,A)\simeq\mathrm{C}^{gm}\text{.}%
\]
On the other hand, $\mathrm{C}$ is truly bigger than $\mathbf{DM}%
_{Nis}^{eff,-}(k,A)$ because it also contains the complexes which are not
bounded from above.
\end{remark}

%\ref{lemma_FFWC}.

\subsection{{Functorial factorizations in $\mathcal{C}^{gm}$}}\label{functorial}

{
Recall from Subsection \ref{prelim} that we may regard $\mathcal{C}^{gm}$ as a Waldhausen category such that its weak
equivalences are the quasi-isomorphisms of the DG structure and its
cofibrations are the cofibrations of Proposition \ref{prop_ModelStruct}.
}

%Now there are several ways to think about the homotopy category
%$Ho(\mathcal{C})$. On the one hand, we may take the homotopy category using
%the DG\ structure on $\mathcal{C}$, which following our conventions we denote
%by roman $\mathrm{C}$.
%
%Alternatively, we may regard $\mathcal{C}$ as a Waldhausen category as in
%Corollary \ref{cor_WaldhausenViewpoint}. We get an (a priori different)
%homotopy category%
%\[
%Ho_{(\operatorname*{Waldhausen})}(\mathcal{C}):=\mathcal{C}[W^{-1}]
%\]
%from the weak equivalences using the formalism of relative categories (see
%\cite{MR2877401}, \cite{MR2877402}). This means that we invert all zig-zags of
%weak equivalences. It is known that this formalism is equivalent to the theory
%of $(\infty,1)$-categories, but in general due to difficulty to control the
%zig-zags, is a formalism which is difficult to work with (see \cite{MR2877401}%
%). However, as the model structure on $\mathcal{C}$ is taken from the DG
%structure, we basically just talk about two different incarnations of an
%$\infty$-categorical enhancement of the triangulated stucture. We record this
%as follows.

\begin{remark}
\label{rmk_obvious}We record the following, essentially obvious, facts.

\begin{enumerate}
\item Any square in $\mathcal{C}$%
\[
\xymatrix{
A \ar[r] \ar[d] & B \ar[d] \\
C \ar[r] & D
}
\]
such that the induced triangle $A\longrightarrow B\oplus C\longrightarrow
D\longrightarrow A[1]$ is distinguished in the triangulated category
$\mathrm{C}$ is weakly equivalent by zig-zags to a genuinely Cartensian square
of complexes in $\mathcal{C}$.

\item $Ho_{(\operatorname*{Waldhausen})}(\mathcal{C})\simeq\mathrm{C}$ as
triangulated categories.
\end{enumerate}
\end{remark}

%We note that $\mathcal{C}^{gm}$ is a full subcategory of $\mathcal{C}$. If we
%consider $\mathcal{C}$ as a Waldhausen category, as in Corollary
%\ref{cor_WaldhausenViewpoint}, we can alternatively characterize%
%\[
%\mathcal{C}^{gm}\subset\mathcal{C}%
%\]
%as the full subcategory {(in a classical and thus unenhanced sense)} whose objects are weakly equivalent in the Waldhausen structure
%to objects whose image in the homotopy category $\mathrm{C}=\mathrm{DM}%
%_{Nis}^{eff}(k;A)$ lies in its thick subcategory $\mathrm{DM}_{gm,Nis}%
%^{eff}(k;A)$. This follows essentially by construction, see Equation
%\ref{lvv1}. We may ask to what extent $\mathcal{C}^{gm}$ respects the
%Waldhausen structure on $\mathcal{C}$. We find the following.
%\begin{lemma}
%\label{lemma_CgmIsWaldhausenSubcat}$\mathcal{C}^{gm}$ is a closed Waldhausen
%subcategory of $\mathcal{C}$.
%\end{lemma}
%
%\begin{proof}
%Every pushout of a cofibration by an arbitrary morphism gives rise to a square%
%\[
%\xymatrix{
%A \ar@{^{(}->}[r] \ar[d] & B \ar[d] \\
%C \ar@{^{(}->}[r] & B \cup_{A} C
%}
%\]
%in $\mathcal{C}$, and so if $A,B,C\in\mathcal{C}^{gm}$, so is the pushout by
%Remark \ref{rmk_obvious}. Moreover, we also need to check that the cokernels
%of cofibration sequences $A\hookrightarrow B\rightarrow C$ in $\mathcal{C}$
%with $A,B\in\mathcal{C}^{gm}$ also have $C\in\mathcal{C}^{gm}$, which holds
%for the same reason. This shows that $\mathcal{C}^{gm}$ is a Waldhausen
%subcategory, and from the definition it is clear that $\mathcal{C}^{gm}$ is
%closed under weak equivalences in $\mathcal{C}$. This proves the claim.
%\end{proof}

\begin{lemma}
\label{lemma_FFWC}$\mathcal{C}^{gm}$ has functorial factorization of weak
cofibrations (FFWC).
\end{lemma}

\begin{proof}
First, we consider $\mathcal{C}$. We use that the Waldhausen structure comes
from a cofibrantly generated model structure (Prop. \ref{prop_ModelStruct}). { In particular, by the small object argument, we have functorial factorizations in this model category.

Its functorial
factorizations into cofibrations followed by acyclic fibrations can be written
as a functor%
\[
\varphi\colon\operatorname*{Fun}([1],\mathcal{C})\longrightarrow
\operatorname*{Fun}\nolimits^{c,w}([2],\mathcal{C})\text{,}%
\]
using the notation of \cite[Appendix A]{MR3993931} (the superscript $c,w$ just
means that the first arrow is a cofibration and the second a weak
equivalence). For FFWC we only need functorial factorizations for weak
cofibrations, i.e. those which are through a zig-zag equivalent to a genuine
cofibration), but $\varphi$ solves the problem even for arbitrary maps. Hence, $$\mod(\mathcal{A}^{\rm op})$$ has FFWC where $\mathcal{A}$ denotes a DG category.

The mere restriction of $\varphi$ to functors with values in
$\mathcal{C}^{gm}$ solves the problem and gives%
\[
\varphi^{\prime}\colon\operatorname*{Fun}([1],\mathcal{C}^{gm})\longrightarrow
\operatorname*{Fun}\nolimits^{c,w}([2],\mathcal{C}^{gm})
\]
For any factorization in the model category of right DG modules made by the above $\varphi$, but
with $A,B\in\mathcal{C}^{gm}$,
\[
A\overset{i}{\hookrightarrow}T\overset{p}{\longrightarrow}B
\]
we get $T\in\mathcal{C}^{gm}$ because $A$ is cofibrant and $A \hookrightarrow T$ is a cofibration, and thus
$$0 \hookrightarrow A \hookrightarrow T$$
is also a cofibration. This shows that $T$ is cofibrant. Perfectness of $T$ follows from the fact that $p$ is a weak equivalence (i.e. quasi-isomorphism), and the fact that $B$ is perfect by assumption.
}
\end{proof}

\subsection{Motivic realization functor}

For any scheme $T$ of finite type over $k$, there is the sheaf of
\emph{equidimensional cycles} $z_{equi}(T,0)\in Sh_{Nis}(Cor_{k})$,
\cite[Definition 16.1]{MR2242284}. { There are also sheaves
$z_{equi}(T,r)$ for $r\in\mathbb{Z}$, but we only need those for $r=0$ and
chose to keep the second parameter in order to remain fully compatible to the
literature.}

\begin{remark}
The notation $Sh_{Nis}(Cor_{k})$ is in line with the cited book
\cite{MR2242284}. As we shall later see (in the proof of Theorem
\ref{thm_MixMotRealization}) that these objects define geometric motives,
there is no problem to regard them as objects in the DG\ category
$\mathcal{C}^{gm}=\mathcal{DM}_{gm,Nis}^{eff}(k;A)$ of the
Beilinson--Vologodsky framework. Before having proven this, one may also
consider them as objects in the DG category of presheaves $PSh_{tr}^{A}%
(Sm_{k})$ considered in \cite[\S 2.2]{MR3004172}.
\end{remark}

We recall the construction. Suppose $S\in Sm_{k}$ and $Z\subseteq S\times
_{k}T$ is an irreducible closed subscheme. We introduce a condition:

\begin{description}
\item[$(\sharp)$] We say that $Z$ satisfies $(\sharp)$ if $Z$ is dominant over
some irreducible component of $S$, and moreover for each $s\in S$ the
scheme-theoretic fiber $Z_{s}=Z\times_{S}\kappa(s)$ is a finite $\kappa(s)$-scheme.
\end{description}

Property $(\sharp)$ is equivalent to demanding that the composed morphism%
\[
\xymatrix{
Z \ar@{^{(}->}[r] \ar[dr] & S \times_{k} T \ar[d] \\
& S
}
\]
is equidimensional of relative dimension zero in the sense of \cite[Def.
2.1.2]{MR1764199}.

For any $S\in Sm_{k}$ define
\[
z_{equi}(T,0)(S)=A\left[  Z\subseteq S\times_{k}T\mid Z\text{ satisfies
}(\sharp)\right]  \text{.}%
\]
The notation refers to the free abelian group having the $[Z]$ as an $A$-basis.

Then $S\mapsto z_{equi}(T,0)(S)$ is a Nisnevich sheaf with transfers, so we
have $z_{equi}(T,0)\in Sh_{Nis}(Cor_{k})$. We need the following maps:

\begin{enumerate}
\item Suppose $i:T^{\prime}\hookrightarrow T$ is a closed immersion. Then
there is an induced map%
\[
i_{\ast}:z_{equi}(T^{\prime},0)(S)\longrightarrow z_{equi}(T,0)(S)
\]
for every $S$ and thus a corresponding morphism of sheaves.\ Moreover, if
$i_{1},i_{2}$ are two composable closed immersions, we have%
\[
i_{2\ast}\circ i_{1\ast}=(i_{2}\circ i_{1})_{\ast}\text{.}%
\]

\item Suppose $j:T^{\prime\prime}\hookrightarrow T$ is an open immersion. Then
there is an induced map%
\[
j^{\ast}:z_{equi}(T,0)(S)\longrightarrow z_{equi}(T^{\prime\prime},0)(S)
\]
for every $S$ and thus a corresponding morphism of sheaves, and for composable
open immersions, we have%
\[
j_{1}^{\ast}\circ j_{2}^{\ast}=(j_{2}\circ j_{1})^{\ast}\text{.}%
\]

\end{enumerate}

All these properties are discussed briefly in \cite[Lec. 16]{MR2242284}. For
(1), the details can be found in \cite[Corollary 3.6.3]{MR1764199}, for
$z_{equi}$ and noting that closed immersions are of course proper. For (2), in
the paragraph before \cite[Prop. 3.6.5]{MR1764199}, using that open immersions
are of course flat.

\begin{remark}
It is perhaps worth to stress that while the sheaves $z_{equi}(T,-)$ are
sheaves on smooth schemes $Sm_{k}$, the scheme $T$ can be \emph{any} $k$-variety
in the sense of \S \ref{sect_RunningConventions}, no matter how singular. This
is important because smooth varieties do not form an $SW$-category
\cite[Remark 3.17]{MR3955537}.
\end{remark}

Now we define a weakly $W$-exact functor in the sense of \cite[Def.
2.17]{MR3993931}. Suppose $\mathcal{V}_{k}$ denotes the $SW$-category of
$k$-varieties (recall the convention on what this entails, see
\S \ref{sect_RunningConventions}).

This is properly defined in \cite[Corollary 3.16, $\mathbf{Var}_{/k}$%
]{MR3955537}, and we quickly recall that (1) the cofibrations $\mathbf{co}%
(\mathcal{V}_{k})$ are the closed immersions (indicated by the arrows style
$\hookrightarrow$), the complement maps $\mathbf{comp}(\mathcal{V}_{k})$ are
open immersions (indicated by the arrow style $\overset{\circ}{\longrightarrow
}$), and subtraction sequences are those of the shape%
\[
Z\hookrightarrow X\overset{\circ}{\longleftarrow}Y\text{,}%
\]
where $Y=X\setminus Z$. As varieties by the running conventions are required
to be reduced, note that $Z$ only appears with its reduced structure and no
nil-thickenings can appear within $\mathcal{V}_{k}$.

Now \cite[Corollary 3.16]{MR3955537} shows that $\mathcal{V}_{k}$ is a
subtractive category and this equips it with the structure of an $SW$-category
if we choose $k$-isomorphisms of $k$-varieties as the weak equivalences
$\mathbf{w}(\mathcal{V}_{k})$.%\footnote{It seems important that these are
%$k$-morphisms, i.e. that they are constant on the base field $k$.}

A weakly $W$-exact functor%
\[
F\colon\mathcal{V}_{k}\longrightarrow\mathcal{C}^{gm}%
\]
is given by a triple $(F_{!},F^{!},F^{w})$ of functors, which agree on objects
but differ on morphisms. The definition is given in \cite[Def. 2.17]%
{MR3993931}.

\begin{remark}
Just like formalization of mixed motives using the category $Cor_{k}$ encodes
the existence of a covariant functoriality alongside a contravariant
functoriality, the axiomatization of weakly $W$-exact functors just uses the
three functors $F_{!},F^{!},F^{w}$ to differentiate between various types of
co- and contravariant transfers.
\end{remark}

We recall that for any presheaf $\mathcal{F}$ of abelian groups, there is a
simplicial presheaf%
\[
(sC)_{n}(\mathcal{F})(X)=\mathcal{F}(X\times\Delta^{n})
\]
enforcing $\mathbb{A}^{1}$-homotopy invariance. A precise definition is given
in \cite[Def. 2.14]{MR2242284} or with more details \cite[\S 4]{MR1764201}. To
$sC_{\bullet}$ one also attaches a complex of presheaves, concentrated in
non-positive degrees, and call it $C_{\ast}$ (It is not really the same, but
quasi-isomorphic to the complex which the Dold--Kan correspondence would
attach to $sC_{\bullet}$. In particular, it is a non-negatively indexed
complex in homological indexing, i.e. it is bounded from \textit{above} in our
cohomological indexing). The construction is functorial in morphisms of
presheaves. We note that if $0\rightarrow\mathcal{F}^{\prime}\rightarrow
\mathcal{F}\twoheadrightarrow\mathcal{F}^{\prime\prime}\rightarrow0$ is an
exact sequence of presheaves, then%
\[
0\rightarrow C_{\ast}\mathcal{F}^{\prime}\rightarrow C_{\ast}\mathcal{F}%
\twoheadrightarrow C_{\ast}\mathcal{F}^{\prime\prime}\rightarrow0
\]
is an exact sequence of complexes. Moreover, if the input $\mathcal{F}$ is a
presheaf with transfers, so is $C_{\ast}\mathcal{F}$.

The main idea behind $C_{\ast}$ is that it can be regarded as the universal
construction enforcing $\mathbb{A}^{1}$-homotopy invariance on a sheaf,
\cite[Example 2.20]{MR2242284}.

We define%
\begin{align}
F_{?}\colon\mathcal{V}_{k}  &  \longrightarrow\mathcal{C}^{gm}\label{lwwwa7}\\
T  &  \longmapsto C_{\ast}z_{equi}(T,0)\nonumber
\end{align}
on objects, and $F_{!}\colon\mathbf{co}(\mathcal{V}_{k})\rightarrow
\mathcal{C}^{gm}$ is just sending closed immersions $i$ to $i_{\ast}$ of
$z_{equi}$, and $F^{!}\colon\mathbf{comp}(\mathcal{V}_{k})\rightarrow
\mathcal{C}^{gm}$ sends open immersions $j$ to $j^{\ast}$ of $z_{equi}$ (as
$i_{\ast}$, $j^{\ast}$ are morphisms of sheaves, they define morphisms in
$\mathcal{C}$). Finally, $F^{w}:\mathbf{w}(\mathcal{V}_{k})\rightarrow
\mathbf{w}(\mathcal{C}^{gm})$ sends a $k$-isomorphism $f:X\rightarrow X$ to
its pushforward $f_{\ast}$ on $z_{equi}$. As $f^{-1}$ exists, $(f^{-1})_{\ast
}$ is a strict inverse $(f^{-1})_{\ast}f_{\ast}=f_{\ast}(f^{-1})_{\ast
}=\operatorname*{id}_{z_{equi}}$, and in particular this is an isomorphism of
sheaves, and therefore (trivially) a weak equivalence in $\mathcal{C}$. This
completes defining $(F_{!},F^{!},F^{w})$ and settles axioms (1)-(4) of a
weakly $W$-exact functor.

\begin{theorem}
\label{thm_MixMotRealization}Suppose $k$ is a perfect field. Let $A$ be a
commutative unital ring. If $k$ has positive characteristic $p>0$, we assume
that $\frac{1}{p}\in A$. Then $F=(F_{!},F^{!},F^{w})$ is a weakly $W$-exact
functor from $\mathcal{V}_{k}$ to $\mathcal{C}^{gm}$. 
{In particular, it induces a map of spectra}
\begin{equation}\label{lafs2}%
K(F)\colon K(\mathcal{V}_{k})\longrightarrow K(\mathcal{C}^{gm})\text{.}
\end{equation}
Moreover, for any
variety $X\in\mathcal{V}_{k}$ we have%
\begin{equation}
F(X)=M^{c}(X) \label{lafs1}%
\end{equation}
in the homotopy category $\mathrm{C}^{gm}=\mathrm{DM}_{gm,Nis}^{eff}(k;A)$,
i.e. $F(X)$ represents the motive with compact support attached to $X$. 
\end{theorem}

If resolution of singularities gets established in positive characteristic,
the theorem also holds without inverting $p$.

\begin{proof}
For the sake of legibility, we give the proof for exponential characteristic
$p=1$, and only comment on the necessary changes if $p>1$. First of all,
regarding Equation \ref{lwwwa7} we need to check that $C_{\ast}z_{equi}(T,0)$
lies in $\mathcal{C}^{gm}$ at all. In view of the definition of geometric
motives in Equation \ref{lvv1} we need to show that its image in the homotopy
category lies in $\mathrm{DM}_{gm,Nis}^{eff}$, but this is the subject of
\cite[Corollary 16.17]{MR2242284}, and furthermore $C_{\ast}z_{equi}(T,0)$
represents $M^{c}(X)$ in the notation of the theory of mixed motives,
\cite[Def. 16.13]{MR2242284} (of course the cited source shows this in the
setting of Voevodsky's framework, but by Remark \ref{rmk_CompareVoevodsky}
this is equivalent to the Beilinson--Vologodsky variant.) It remains to verify
axioms (5)-(7) of a weakly $W$-exact functor. Given the diagram below on the
left,%
\[
\xymatrix{
X \ar[d]_{j}^{\circ} \ar@{^{(}->}^{i}[r] & Z \ar[d]^{j^{\prime} }_{\circ}  \\
Y \ar@{^{(}->}_{i^{\prime}}[r] & W
}
\qquad\qquad
\xymatrix{
F(X) \ar[r]^{i_{\ast} } & F(Z)   \\
F(Y) \ar[u]^{j^{\ast}} \ar[r]_{i^{\prime}_{\ast} } & F(W) \ar[u]_{j^{\prime
\ast} }
}
\]
and assuming this is a cartesian diagram in $\mathcal{V}_{k}$, axiom (5)
demands that the strict functorialities result in the diagram above on the
right. However, this is the statement of \cite[Prop. 3.6.5]{MR1764199} in the
special case of $z_{equi}$, the pushforward proper (called $p$ loc. cit., and
with $d=0$ in the notation of the cited proposition). Next, given a
subtraction sequence%
\[
\xymatrix{
Z \ar@{^{(}->}[r]^{i} & X \\
& X - Z  \ar[u]^{j}_{\circ}
}
\]
in $\mathcal{V}_{k}$, axiom (6) is equivalent to demanding that%
\[
C_{\ast}z_{equi}(Z)\overset{i_{\ast}}{\longrightarrow}C_{\ast}z_{equi}%
(X)\overset{j^{\ast}}{\longrightarrow}C_{\ast}z_{equi}(X-Z)\longrightarrow
C_{\ast}z_{equi}(Z)[1]
\]
is a distinguished triangle in $\mathrm{DM}_{gm,Nis}^{eff}$, because if it is,
the square stated in axiom (6) will be weakly cocartesian. This is the first
part of \cite[Theorem 5.11]{MR1764201} if the field $k$ admits a resolution of
singularities.\ For positive characteristic, work of Kelly can be used as a
replacement when using $A\left[  \frac{1}{p}\right]  $-coefficients
\cite{MR3673293}. This is why in positive characteristic, we simply assume
that $p$ is already invertible in $A$. Axiom (7) really makes two statements:
The one for $F_{!}$ follows directly from the strict functoriality of the
pushforward along closed immersions as our weak equivalences are also induced
by a pushforward. The other part, for $F^{!}$, is a special case of Axiom (5).
This settles the proof if $k$ has resolution of singularities. As explained,
for $p>1$ and as long as this is not the case (if it ever will), we may also
invert $p$ everywhere, i.e. in all the categories in this section, and use
Kelly's foundations \cite{MR3673293}. Having settled that $F\ $is weakly
$W$-exact, \cite[Proposition 2.19]{MR3993931} produces a map of $K$-theory spectra.
\end{proof}

\begin{variant}
One could also define a variant%
\begin{align*}
\tilde{F}_{?}\colon\mathcal{V}_{k}  &  \longrightarrow\mathcal{C}^{gm}\\
T  &  \longmapsto z_{equi}(T,0)
\end{align*}
without using $C_{\ast}$. As $z_{equi}$ is, regarded as a complex,
concentrated in the single degree zero, we may employ \cite[Lemma
14.4]{MR2242284}, which tells us that there is an isomorphism%
\begin{equation}
z_{equi}(T,0)\overset{\sim}{\longrightarrow}C_{\ast}(z_{equi}(T,0))
\label{lwwwa7a}%
\end{equation}
in $\mathrm{DM}_{gm,Nis}^{eff}$. A variation of the above proof goes through,
but one has to plug in the equivalence of Equation \ref{lwwwa7a} and exploit
its functoriality in the verification of the axioms (6) and (7).
\end{variant}

\section{Realization maps}

\subsection{Generalities on realizations for motives}

We shall employ Vologodsky's technique to produce realizations of mixed
motives, as developed in his paper \cite{MR3004172}. The idea is that any
functor defined on smooth $k$-varieties and satisfying good descent properties
automatically extends \textit{uniquely} to a functor on \'{e}tale mixed motives.

We recall the idea in some more detail. If $\mathcal{C}_{1},\mathcal{C}_{2}$
are DG categories, write $\mathcal{T}(\mathcal{C}_{1},\mathcal{C}_{2})$ for
the category of DG quasi-functors%
\[
H\colon\mathcal{C}_{1}\longrightarrow\mathcal{C}_{2}\text{.}%
\]
Any such functor induces a functor of homotopy categories%
\[
Ho(H)\colon Ho(\mathcal{C}_{1})\longrightarrow Ho(\mathcal{C}_{2})
\]
and if $\mathcal{C}_{1},\mathcal{C}_{2}$ are triangulated, this is a
triangle functor of triangulated categories {(see \cite[Ex. 5.1.4]{toen})}. 
We write%
\[
\mathcal{T}^{c}(\mathcal{C}_{1},\mathcal{C}_{2})\subseteq\mathcal{T}%
(\mathcal{C}_{1},\mathcal{C}_{2})
\]
for the full subcategory of DG quasi-functors $H$ such that $Ho(H)$ commutes
with arbitrary direct sums.

If $\mathcal{C}$ is a DG category, $\underrightarrow{\mathcal{C}}$ denotes the
DG ind-completion. This is a triangulated cocomplete DG category such that
$\mathcal{C}\subseteq(\underrightarrow{\mathcal{C}})^{perf}$ and defining a DG
quasi-functor on $\underrightarrow{\mathcal{C}_{1}}$ commuting with all direct
sums is equivalent to providing a DG quasi-functor only on $\mathcal{C}_{1}$.
\[
\mathcal{T}^{c}(\underrightarrow{\mathcal{C}_{1}},\mathcal{C}_{2}%
)\overset{\sim}{\longrightarrow}\mathcal{T}(\mathcal{C}_{1},\mathcal{C}_{2})
\]

Now take $\mathcal{C}_{1}\colon=A[Sm]$, the category whose objects are smooth
$k$-varieties (denoted by $A[X]$ for $X$ the smooth $k$-variety) and morphisms
are%
\[
\operatorname*{Hom}\nolimits_{A[Sm]}(A[\underset{i}{\sqcup}X_{i}%
],A[Y])\colon=\coprod_{i}A[\operatorname*{Hom}\nolimits_{k}(X_{i},Y)]
\]
for each $X_{i}$ connected. See \cite[\S 2.4]{MR3004172}.

Vologodsky now introduces the full subcategory $\mathcal{T}^{h,\Delta}(-)$,%
\[
\mathcal{T}^{h,\Delta}(A[Sm],\mathcal{C}_{2})\subseteq\mathcal{T}%
(A[Sm],\mathcal{C}_{2})
\]
consisting of those DG quasi-functors $H$ which

\begin{description}
\item[$\left(  h\right)  $] satisfy hyperdescent with respect to the
$h$-topology (i.e. if $U_{\bullet}\rightarrow X$ is an $h$-hypercovering, then
$H$ sends%
\[
A_{tr}[U_{\bullet}]\longrightarrow A_{tr}[X]
\]
to a quasi-isomorphism),

\item[$\left(  \Delta\right)  $] satisfy $\mathbb{A}^{1}$-invariance (i.e. $H$
sends%
\[
A_{tr}[X\times\mathbb{A}_{k}^{1}]\longrightarrow A_{tr}[X]
\]
to a quasi-isomorphism).
\end{description}

\begin{theorem}
[{Vologodsky, \cite[Theorem 2]{MR3004172}}]\label{thm_Volo}Suppose
$\mathcal{C}$ is a triangulated cocomplete and compactly genererated
$A$-linear DG category. There is an equivalence of categories%
\[
\Phi\colon\mathcal{T}^{c}(\mathcal{DM}_{\acute{e}t}^{eff}(k;A),\mathcal{C}%
)\overset{\sim}{\longrightarrow}\mathcal{T}^{h,\Delta}(A[Sm],\mathcal{C}%
)\text{,}%
\]
such that if $X$ is a smooth proper variety and $M(X)$ its motive, then for any $H \in \mathcal{T}^{h,\Delta}(A[Sm],\mathcal{C})$%
\[
\Phi^{-1}(H)(M(X))=H(X)\text{.}%
\]

\end{theorem}

We just state this as the existence of $\Phi$, but really the equivalence is
constructed in a concrete fashion in \cite{MR3004172}. We will not recall this
in detail.

In other words: As soon as we exhibit a DG quasi-functor $H$ defined on
$A[Sm]$ which satisfies $\left(  h\right)  $ and $\left(  \Delta\right)  $,
this uniquely determines an extension $\Phi^{-1}(H)$ on all \'{e}tale
effective mixed motives.

Vologodsky then constructs all the standard realizations of mixed motives by
first setting them up on $A[Sm]$ (the details are a little more involved, but
this is the essential point).

We can now deduce the corresponding realization in the setting of this paper.

\subsection{The realization theorem}

Suppose $k$ is a field for which resolution of singularities is available,
e.g. of characteristic zero. As before, let $A$ be a commutative unital ring.
If $p$ has positive characteristic $p>0$, we assume $\frac{1}{p}\in A$.

{Suppose $\mathcal{D}$ is a triangulated cocomplete and compactly
genererated $A$-linear DG category. 
Let%
\[
\mathcal{D}_{finite}\subseteq\mathcal{D}%
\]
some full DG\ subcategory. Usually, $\mathcal{D}_{finite}:=\mathcal{D}^{perf}$
will the choice we are interested in; the full DG\ subcategory of objects
whose image in the homotopy category $\mathrm{D}$ is
compact. As discussed earlier, we can associate study the algebraic $K$-theory $K(\mathcal{D}_{finite})$ as the Waldhausen $K$-theory of the Waldhausen category $\mathrm{perf}_{co}(\mathcal{D}_{finite})$. We remind the reader of Remark \ref{convention}, according to which we do not notationally distinguish between an object in $\mathcal{D}_{finite}$ and the induced DG module (given by the Yoneda embedding).}
\begin{theorem}
\label{thm_GeneralRealizationTheorem}Suppose $k$ is a perfect field having
resolution of singularities and of finite cohomological dimension with respect
to $A$-coefficients in the sense of Equation \ref{l_Cond_GaloisCohomDim}.
Suppose we are given a DG quasi-functor%
\[
\phi\in\mathcal{T}^{h,\Delta}(A[Sm],\mathcal{D}_{finite})\text{.}%
\]
Then there is a map of spectra%
\[
\operatorname*{R}\nolimits^{\phi}\colon K(\mathcal{V}_{k})\longrightarrow
K(\mathcal{D}_{finite})
\]
such that the following hold:

\begin{enumerate}
\item Suppose $X$ is a smooth proper $k$-variety. In $K_{0}$ we get%
\[
\operatorname*{R}\nolimits^{\phi}([X])=[\phi(X)]
\]
and if $f\colon X\overset{\sim}{\rightarrow}X$ is any automorphism, we get in
$K_{1}$ that%
\[
\operatorname*{R}\nolimits^{\phi}([f\colon X\overset{\sim}{\rightarrow
}X])=\left[  \phi(f)\colon\phi(X)\overset{\sim}{\rightarrow}\phi(X)\right]
\text{.}%
\]

\item Suppose $X$ is a smooth $k$-variety, $\overline{X}$ a smooth
compactification with a smooth closed subvariety $Z\subseteq\overline{X}$ such
that $X=\overline{X}\setminus Z$. Then in $K_{0}$,%
\[
\operatorname*{R}\nolimits^{\phi}([X])=[\phi(\overline{X})]-[\phi(Z)]
\]
and if one can extend the automorphisms such that%
\[
\xymatrix{
Z \ar[d]_{f} \ar@{^{(}->}[r] & \overline{X} \ar[d]^{f} \\
Z \ar@{^{(}->}[r] & \overline{X}
}
\]
commutes, then $\operatorname*{R}\nolimits^{\phi}([f\colon X\overset{\sim
}{\rightarrow}X])=\operatorname*{R}\nolimits^{\phi}([f\colon\overline
{X}\overset{\sim}{\rightarrow}\overline{X}])-\operatorname*{R}\nolimits^{\phi
}([f\mid_{Z}\colon Z\overset{\sim}{\rightarrow}Z])$.
\end{enumerate}
\end{theorem}

\begin{proof}
By Theorem \ref{thm_MixMotRealization} we obtain a map of spectra%
\[
K(F)\colon K(\mathcal{V}_{k})\longrightarrow K(\mathcal{C}^{gm})\text{.}%
\]
The DG\ quasi-functor $\phi\in\mathcal{T}^{h,\Delta}(A[Sm],\mathcal{D}%
_{finite})$ defines (tautologically) an element of $\mathcal{T}^{h,\Delta
}(A[Sm],\mathcal{D})$ by $\mathcal{D}_{finite}\subseteq\mathcal{D}$. From
Vologodsky's Theorem (Theorem \ref{thm_Volo}) we therefore obtain a
DG\ quasi-functor%
\[
\tilde{\phi}\in\mathcal{T}^{c}(\mathcal{DM}_{\acute{e}t}^{eff}%
(k;A),\mathcal{D})\text{,}%
\]
i.e.%
\[
\tilde{\phi}\colon\mathcal{DM}_{\acute{e}t}^{eff}(k;A)\longrightarrow
\mathcal{D}\text{.}%
\]
Note that the values of $\tilde{\phi}$ need \textit{not} lie in $\mathcal{D}%
_{finite}$ because of the colimit procedures carried out in the construction
of $\tilde{\phi}$. However, we have%
\[
\tilde{\phi}(M(X))=\phi(X)\in\mathcal{D}_{finite}%
\]
if $X$ is a smooth $k$-variety and $M(X)$ its motive, still by Theorem
\ref{thm_Volo}. As the triangulated category $\mathrm{DM}_{\acute{e}t}%
^{eff}(k;A)$ of geometric motives is compactly generated and the compact
objects are generated by the $M(X)$ (\cite[Corollary 2.4 and below]%
{MR3004172}), we obtain a DG quasi-functor to $\mathcal{D}_{finite}$,%
\[
\left.  \tilde{\phi}\mid_{\mathcal{DM}_{gm,\acute{e}t}^{eff}(k;A)}%
\colon\right.  \mathcal{DM}_{gm,\acute{e}t}^{eff}(k;A)\longrightarrow
\mathcal{D}_{finite}\text{.}%
\]
Now form the composition of DG\ quasi-functors%
\[
\mathcal{C}^{gm}=\mathcal{DM}_{gm,Nis}^{eff}(k;A)\overset{(\#)}%
{\longrightarrow}\mathcal{DM}_{gm,\acute{e}t}^{eff}(k;A)\overset{\tilde{\phi}%
\mid_{(...)}}{\longrightarrow}\mathcal{D}_{finite}\text{,}%
\]
where $(\#)$ is the change-of-topology quasi-functor from Equation \ref{lzz1}.
Suppose $X$ is a smooth \textit{proper} $k$-variety. Then there is a canonical
quasi-isomorphism $M^{c}(X)\underset{qis}{\overset{\sim}{\longrightarrow}%
}M(X)$ from the compactly supported motive, see \cite[Lec. 16]{MR2242284}. Now
we can prove (1). On the level of $K_{0}$, we have%
\[
K(F)_{0}\colon K_{0}(\mathcal{V}_{k})\longrightarrow K_{0}(\mathcal{C}%
^{gm})\longrightarrow K_{0}(\mathcal{D}_{finite})
\]
sending%
\[
\lbrack X]\mapsto\lbrack M^{c}(X)]=[M(X)]\mapsto\lbrack\tilde{\phi}(X)]
\]
and on $K_{1}$ we can basically use the same argument. We now prove (2). Using
\cite[Theorem 16.15]{MR2242284} we have a triangle%
\[
M^{c}(Z)\overset{i_{\ast}}{\longrightarrow}M^{c}(\overline{X})\overset
{j^{\ast}}{\longrightarrow}M^{c}(X)\longrightarrow M^{c}(Z)[1]
\]
and thus a canonical quasi-isomorphism between $M^{c}(X)$ and the cone (on the
DG category level) below on the left,%
\begin{equation}
\operatorname{cone}\left(  M^{c}(Z)\overset{i_{\ast}}{\longrightarrow}%
M^{c}(\overline{X})\right)  \overset{\sim}{\longrightarrow}_{qis}%
\operatorname{cone}\left(  M(Z)\overset{i_{\ast}}{\longrightarrow}%
M(\overline{X})\right)  \text{,} \label{lafs2b}%
\end{equation}
and the quasi-isomorphism in the middle holds since $Z$ and $\overline{X}$ are
both smooth and proper. Hence, we can reduce our claim about $K_{0}$ to (1)
and the cone of complexes is easy to see to correspond to a difference in
$K_{0}$. For the claim about $\varphi$, note that Equation \ref{lafs1a}
ensures that $\varphi$ also acts on the cones in Equation \ref{lafs2b}, and
then also reduces to (1).\newline
\end{proof}

\section{Example: The Betti realization}

We illustrate our constructions with the concrete example of the Betti
realization. Suppose $k\subseteq\mathbb{C}$ is any subfield of the complex
numbers. If $X$ is a $k$-variety, we write $X_{\mathbb{C}}$ for the complex
manifold attached to the smooth $\mathbb{C}$-variety $X\times_{k}%
\operatorname*{Spec}\mathbb{C}$.

Suppose $A$ is a commutative Noetherian unital ring, which will
serve as the coefficient ring for the Betti realization. We write
$C^{\operatorname{sing}}(X_{\mathbb{C}},A)$ for the singular chain complex
with coefficients in $A$.

{\begin{remark}
{One can drop Noetherian if
one works with coherent $A$-modules instead of finitely generated ones below. We prefer to stick to the Noetherian assumption for the sake of simplicity.}
\end{remark}}

\begin{theorem}
There is a map of spectra%
\[
\operatorname*{R}\nolimits_{A}^{Betti}\colon K(\mathcal{V}_{k})\longrightarrow
K(\mathsf{Mod}_{fg}(A))
\]
such that the following hold:

\begin{enumerate}
\item Suppose $X$ is a smooth proper $k$-variety. In $K_{0}$ we get%
\begin{equation}
\operatorname*{R}\nolimits_{A}^{Betti}([X])=\sum_{i}(-1)^{i}[H^{Betti}%
_{i}(X_{\mathbb{C}},A)]\text{,} \label{lzz3}%
\end{equation}
the cohomology of the complex manifold $X_{\mathbb{C}}$. If $\varphi\colon
X\overset{\sim}{\rightarrow}X$ is any automorphism, we get%
\[
\operatorname*{R}\nolimits_{A}^{Betti}([\varphi\colon X\overset{\sim
}{\rightarrow}X])=\sum_{i}(-1)^{i}[\varphi_{\ast}H^{Betti}_{i}(X_{\mathbb{C}%
},A)]
\]
in $K_{1}$.

\item Suppose $X$ is a smooth $k$-variety, $\overline{X}$ a smooth
compactification with a smooth closed subvariety $Z\subseteq\overline{X}$ such
that $X=\overline{X}\setminus Z$. Then in $K_{0}$,%
\[
\operatorname*{R}\nolimits_{A}^{Betti}([X])=\operatorname*{R}\nolimits_{A}%
^{Betti}([\overline{X}]-[Z])
\]
and if
\begin{equation}
\xymatrix{
Z \ar[d]_{\varphi} \ar@{^{(}->}[r] & \overline{X} \ar[d]^{\varphi} \\
Z \ar@{^{(}->}[r] & \overline{X}
}
\label{lafs1a}
\end{equation}
commutes, then $\operatorname*{R}\nolimits_{A}^{Betti}([\varphi\colon
X\overset{\sim}{\rightarrow}X])=\operatorname*{R}\nolimits_{A}^{Betti}%
([\varphi\colon\overline{X}\overset{\sim}{\rightarrow}\overline{X}%
]-[\varphi\colon Z\overset{\sim}{\rightarrow}Z])$.
\end{enumerate}
\end{theorem}

\begin{remark}
If $A$ is regular, this yields a map%
\[
\operatorname*{R}\nolimits_{A}^{Betti}\colon K(\mathcal{V}_{k})\longrightarrow
K(A)
\]
with the same properties, except that each $H^{Betti}_{i}$ is tacitly replaced
by a finite projective resolution.
\end{remark}

\begin{proof}
Condition \ref{l_Cond_GaloisCohomDim} is harmless because we can first
base-change%
\[
\mathcal{V}_{k}\longrightarrow\mathcal{V}_{\mathbb{C}}%
\]
and work in the latter situation, where the condition is tautologically
satisfied (Example \ref{example_CohomDim}). So without loss of generality,
$k=\mathbb{C}$. We use Theorem \ref{thm_GeneralRealizationTheorem} for the
Betti realization, as in \cite[\S 2.7]{MR3004172}. A smooth variety $X$ gets
sent to the singular chain complex with coefficients in $A$,%
\[
C^{\operatorname{sing}}(X_{\mathbb{C}},A)\text{.}%
\]
For the motive $M(X)$ of a smooth variety the complex $C^{\operatorname{sing}%
}(X_{\mathbb{C}},A)$ has finite homological support (namely concentrated in
degrees $[0,2\dim X]$) and each cohomology group is finitely generated. Hence,
we may take bounded complexes of finitely generated $A$-modules%
\[
\mathcal{D}_{finite}:=C^{b}(\mathsf{Mod}_{fg}(A))\qquad\text{and}%
\qquad\mathcal{D}:=C(\mathsf{Mod}(A))
\]
inside all complexes and all modules (all with the standard DG structure). If
we use Theorem \ref{thm_GeneralRealizationTheorem} now, we obtain almost our
claim, except for the map being%
\[
K(\mathcal{V}_{k})\longrightarrow K(C^{b}(\mathsf{Mod}_{fg}(A)))
\]
Note that the Waldhausen structure on $C^{b}(\mathsf{Mod}_{fg}(A))$ is such
that the weak equivalences are the quasi-isomorphisms, so within $K$-theory
language one would perhaps stress this by writing $K(qC^{b}(\mathsf{Mod}%
_{fg}(A)))$. By the Gillet--Waldhausen Theorem (\cite[Chapter V, Theorem
2.2]{MR3076731}) there is an equivalence%
\[
K(\mathsf{Mod}_{fg}(A))\longrightarrow K(qC^{b}(\mathsf{Mod}_{fg}(A)))\text{,}%
\]
where the left side is the ordinary $K$-theory of an abelian category and the
right side is the aforementioned Waldhausen $K$-theory with respect to weak
equivalences. The map is induced from sending an object $M\in\mathsf{Mod}%
_{fg}(A)$ to the complex concentrated in degree zero. The inverse map, on the
level of $K_{0}$ (and analogously $K_{1}$), then is%
\[
\lbrack C^{\bullet}]\mapsto\sum_{i}(-1)^{i}[C^{i}]
\]
for any bounded complex $C^{\bullet}$. This yields the statement of our
theorem, in particular the signs in Equation \ref{lzz3}.
\end{proof}

\section{Example: The Hodge realization}

Let $k=\mathbb{C}$ and $\mathbb{Z}\subseteq A\subseteq\mathbb{Q}$. Let
$MHS_{eff}^{A}$ be the category of effective polarizable $A$-Hodge structures
(as in \cite[\S 2.8]{MR3004172}). Vologodsky characterizes {effectiveness} by%
\[
F^{1}=0\text{.}%
\]
{This property is evidently extension-closed, and thus $MHS_{eff}^{A}$ is endowed with the structure of an exact category.}

\begin{remark}
For a pure Hodge structure $M$ of weight $w$ with%
\[
M\otimes_{A}\mathbb{C=}\bigoplus_{p+q=w}M^{p,q}%
\]
effectiveness is equivalent to%
\[
M^{p,q}=0
\]
{for $p>0$ or $q>0$. This is because we follow Vologodsky's conventions, which models the properties of effective pure Hodge structures on Hodge structures on singular \emph{homology} of a smooth projective variety.}
\end{remark}

\begin{theorem}
There is a map of spectra%
\[
\operatorname*{R}\nolimits_{A}^{Hodge}\colon K(\mathcal{V}_{k})\longrightarrow
K(MHS_{eff}^{A})
\]
such that the following hold:

\begin{enumerate}
\item Suppose $X$ is a smooth proper $k$-variety. In $K_{0}$ we get%
\[
\operatorname*{R}\nolimits_{A}^{Hodge}([X])=\sum_{i}(-1)^{i}[H_{i}%
(X_{\mathbb{C}},A)]\text{,}%
\]
the cohomology of the complex manifold $X_{\mathbb{C}}$. If $\varphi\colon
X\overset{\sim}{\rightarrow}X$ is any automorphism, we get%
\[
\operatorname*{R}\nolimits_{A}^{Hodge}([\varphi\colon X\overset{\sim
}{\rightarrow}X])=\sum_{i}(-1)^{i}[\varphi_{\ast}H_{i}(X_{\mathbb{C}},A)]
\]
in $K_{1}$.

\item Suppose $X$ is a smooth $k$-variety, $\overline{X}$ a smooth
compactification with a smooth closed subvariety $Z\subseteq\overline{X}$ such
that $X=\overline{X}\setminus Z$. Then in $K_{0}$,%
\[
\operatorname*{R}\nolimits_{A}^{Hodge}([X])=\operatorname*{R}\nolimits_{A}%
^{Hodge}([\overline{X}]-[Z])
\]
and if%
\[
\xymatrix{
Z \ar[d]_{\varphi} \ar@{^{(}->}[r] & \overline{X} \ar[d]^{\varphi} \\
Z \ar@{^{(}->}[r] & \overline{X}
}
\]
commutes, then $\operatorname*{R}\nolimits_{A}^{Hodge}([\varphi\colon
X\overset{\sim}{\rightarrow}X])=\operatorname*{R}\nolimits_{A}^{Hodge}%
([\varphi\colon\overline{X}\overset{\sim}{\rightarrow}\overline{X}%
]-[\varphi\colon Z\overset{\sim}{\rightarrow}Z])$.
\end{enumerate}
\end{theorem}

\begin{proof}
We proceed as for the Betti realization, but with%
\[
\mathcal{D}_{finite}:=C^{b}(MHS_{eff}^{A})\qquad\text{and}\qquad
\mathcal{D}:=C(\mathsf{Lex}(MHS_{eff}^{A}))\text{,}%
\]
where%
\[
MHS_{eff}^{A}\hookrightarrow\mathsf{Lex}(MHS_{eff}^{A})
\]
is the Quillen embedding, realizing the exact category as an extension-closed
full subcategory of a Grothendieck abelian category. Again, by
Gillet--Waldhausen (as in the plain Betti situation) reduce from bounded
complexes up to quasi-isomorphism to $K(MHS_{eff}^{A})$.
\end{proof}%

\section{Other realizations}

By the work of Cisinski and D\'{e}glise any mixed Weil cohomology (as well as what they call a \textit{stable cohomology} loc. cit.) satisfies $h$-descent and $\mathbb{A}^1$-invariance, see \cite[Corollary 17.2.6]{MR3971240}. This should cover most interesting realizations whose coefficients are a $\mathbb{Q}$-algebra. For example (besides the aforementioned Betti realization if one takes $\mathbb{Q}$-coefficients) de Rham cohomology or syntomic cohomology over $p$-adic fields (\cite{MR3394127}, or use \cite[Theorem A.1]{MR3556797}). We leave carefully choosing categories of values $\mathcal{D}_{finite}$ to the reader, as this varies from case to case, and the optimal details may depend on what concrete applications the reader may have in mind.

Another realization of interest can be obtained from \cite{sosnilo}. In \emph{loc. cit.}, Sosnilo constructs an exact functor of stable $\infty$-categories $\mathcal{C} \to \mathrm{Com}^b(\underline{Hw})$, where $\mathcal{C}$ is a stable $\infty$-category endowed with a bounded weight structure in the sense of Bondarko's \cite{bondarko}. The notation $\underline{Hw}$ refers to the heart of the weight structure. The DG category $\mathcal{DM}_{gm,t}^{eff}(k;A)$ possesses such a weight structure, whenever the exponential characteristic of $k$ is invertible in $A$. In this case, the heart is given by the additive category of effective Chow motives. Therefore, we obtain a realization map
$$K(\mathcal{V}_k) \to K(\mathcal{C}^{gm}) \to K(\mathrm{Com}^b(\mathrm{Chow}^{eff}))\text{.}$$
On the level of $K_0$ this recovers a well-known construction of Gillet--Soul\'e \cite{gs}. Details will appear elsewhere. In \cite[Problem 7.5]{MR3993931}, the authors speculate about the existence of such a \emph{Gillet--Soul\'e realization} and pose several questions about its properties. 

%%% BIBLIOGRAPHY

\end{document}